\documentclass[11pt]{amsart}

\usepackage{latexsym}
\usepackage{amssymb}
\usepackage{amsthm,amsmath,color}
\usepackage{graphicx}

\newtheorem{theorem}{Theorem}[section]

\newtheorem{lemma}[theorem]{Lemma}

\newtheorem{conjecture}[theorem]{Conjecture}
\theoremstyle{definition}

\newtheorem{remark}[theorem]{Remark}

\numberwithin{equation}{section}

\newcommand{\E}{\mathbb{E}}

\newcommand{\R}{\mathbb{R}}
\newcommand{\C}{\mathbb{C}}
\newcommand{\N}{\mathbb{N}}
\newcommand{\Z}{\mathbb{Z}}
\newcommand{\T}{\mathbb{T}} 
\newcommand{\Vol}{\mathrm{Vol}}

\newcommand{\e}{\varepsilon}

\DeclareMathOperator{\tr}{Tr}

\author{Stanislaw J. Szarek and  Pawel Wolff}
\address[S. Szarek]{Case Western Reserve University, Cleveland, Ohio 44106-7058, U.S.A. \ and \ Institut de Math\'ematiques de Jussieu-PRG, 
Sorbonne Universit\'e, 4 place Jussieu, 75252 Paris cedex 05, France}
\email{szarek@cwru.edu}
\address[P. Wolff]{Institut de Math\'ematiques de Toulouse, Universit\'e de Toulouse \& CNRS, UPS, F-31062 Toulouse Cedex 09, France; Current address: ECMWF, Robert-Schuman-Platz 3, 53175 Bonn, Germany}
\email{pawel.wolff@ecmwf.int}



\title{Radii of Euclidean sections of $\ell_p$-balls}

\begin{document}
\begin{abstract}
The celebrated Dvoretzky theorem asserts that every $N$-dimensional convex body admits central sections of dimension $d = \Omega(\log N)$, which are nearly spherical. For many instances of convex bodies,  typically unit balls with respect to some norm, much better lower bounds on $d$ have been obtained, with most research focusing on such lower bounds and on the degree of approximation of the section by a $d$-dimensional Euclidean ball. In this note we concentrate on another parameter, namely the {\sl radius } of the approximating ball. 
We focus on the case of the unit ball of the space $\ell_1^N$ (the so-called {\em cross-polytope}), which is relevant to various questions of interest in theoretical computer science.  We will also survey other instances where similar questions for other normed spaces (most often $\ell_p$-spaces or their non-commutative analogues) were found relevant to problems in various areas of mathematics and its applications, and state some open problems. Finally, in view of the computer science ramifications, we will comment on the algorithmic aspects of finding nearly spherical sections.
\end{abstract}

\maketitle

\section{Introduction} \label{section1}
Consider $\R^N$ with two norms: $\ell_1$ and $\ell_2$. Although for an arbitrary vector $x \in \R^N$ the discrepancy between these two norms can be as large as $\sqrt{N}$, i.e., $\|x\|_2 \le \|x\|_1 \le \sqrt{N} \|x\|_2$ (and these estimates are sharp), it is a classical fact that there exist subspaces of $\R^N$ of dimensions even proportional to $N$ 
 (i.e., $\Omega(N)$), on which the two norms are comparable. More specifically, for any fixed $\varepsilon > 0$ there exists a subspace $E \subseteq \R^N$ of dimension at least $c(\varepsilon) N$ such that, for all $x \in E$,
\begin{eqnarray}
  \label{eq:almost-euclidean-1}
  (1 - \varepsilon)\lambda \|x\|_2 \le \|x\|_1 & \\
  \|x\|_1 & \le (1 + \varepsilon) \lambda \|x\|_2
  \label{eq:almost-euclidean-2}
\end{eqnarray}
for some scaling constant $\lambda > 0$  (\cite{FLM77}). In such case $E$ is called an \emph{almost-Euclidean} subspace of $\ell_1^N$. If one denotes $B_1^N$ and $B_2^N$ the unit balls in $\R^N$ that correspond to the norms of $\ell_1$ and of $\ell_2$ respectively, then~\eqref{eq:almost-euclidean-1} and~\eqref{eq:almost-euclidean-2} can be rephrased as
\begin{equation}
  \label{eq:almost-euclidean-3}
  \frac{\lambda^{-1}}{1+\varepsilon}  B_2^N \cap E \subseteq B_1^N \cap E \subseteq \frac{\lambda^{-1}}{1-\varepsilon}  B_2^N \cap E .
\end{equation}
In other words,  $\lambda^{-1}/(1-\varepsilon)$, which is (an upper bound for) the outer radius of a section of $B_1^N$ with $E$, is remarkably close to the in-radius of that section, which is at least $\lambda^{-1}/(1+\varepsilon)$. 
A companion result to this striking instance of Dvoretzky's theorem \cite{Dvoretzky}, in the \emph{isomorphic} regime, asserts that, for any $\eta \in (0,1)$, 
there exists a subspace $E \subseteq \R^N$ of codimension at most $\eta N$ such that, for all $x \in E$,
\begin{equation}
  \label{eq:nearly-euclidean}
 c(\eta)  \sqrt{N} \, \lambda \|x\|_2 \le  \|x\|_1  \le  \sqrt{N} \|x\|_2,  
  \end{equation}
see \cite{Kashin77}. (Such $E$ is then called a
{\em nearly-Euclidean} subspace.) It needs to be emphasized that proofs of existence of large almost-Euclidean (or nearly-Euclidean) subspaces are generally probabilistic and (except for very special instances, see Section \ref{section4}) do not yield subspaces $E$ that are explicit or can be efficiently exhibited.

In this note we are interested in sharp bounds for the scaling constant $\lambda$ in \eqref{eq:almost-euclidean-1}--\eqref{eq:nearly-euclidean} or, equivalently, for the Euclidean radius of the corresponding section.  
In known constructions of almost-Euclidean subspaces of $\ell_1^N$ of proportional dimension, 
$\lambda$ is of order $\sqrt{N}$, and it follows from \eqref{eq:nearly-euclidean} 
that this a necessity and not an artifact of the method. 
Below (in Theorem \ref{thm:bound-on-radius}) we will make this assertion more precise by showing that $\lambda$ is  between $\sqrt{d}$ 
and (essentially) $\sqrt{\frac{2}{\pi}} \,\sqrt{N}$, where $d$ is the dimension of the subspace.  
 In terms of 
 sections of $B_1^N$-balls, our result asserts that for \emph{ any} $d$-dimensional section the inner radius is at most $d^{-1/2}$ while the outer radius is (essentially, for large $d$) 
 at least $\sqrt{\frac{\pi}{2}} \, N^{-1/2}$.  
 
 The interest in the value of the scaling constant $\lambda$ stems (for example) from the fact that 
 large nearly-Euclidean subspaces of $\ell_1$-spaces have applications in theoretical computer science, including, but not limited
 to, the area of compressed sensing (see, e.g., \cite{Don06}).  Finding explicit examples of such 
 subspaces, or efficient algorithms to generate them is a major project,  
 see \cite{Indyk} and cf. \cite{GHR}.   
 Existence of subspaces verifying \eqref{eq:almost-euclidean-1} with $\lambda$ close to $\sqrt{N}$ 
 would allow (by arguments from \cite{IndykSzarek10}) for constructions of such large subspaces which --- while still being probabilistic --- would require relatively few random bits. Unfortunately, our result supplies a ``no-go theorem.''
 A result in similar vein 
 (and similarly motivated) can be found in \cite{MR13}. 

\smallskip \paragraph{\em Notation.} Above and in what follows, 
 $C, c, c_0$ etc. refer to (strictly) positive constants, independent of the instance of the problem, most notably of the dimension. 
 Different occurrences of the same symbol do not necessarily indicate the same numerical value. Similarly, $c(\delta)$ etc. is a positive constant, which depends {\em only } on the parameter $\delta$.  For comparing different quantities, we will freely use the Bachmann-Landau asymptotic notation, which is standard  (particularly) in theoretical computer science : $O(\cdot)$, $o(\cdot)$, $\Omega(\cdot)$, $\Theta(\cdot)$.  
 For a vector $x \in \R^N$ we refer to its coordinates as $x(1), \ldots, x(N)$. 
By $g_1, g_2, \ldots $ we denote independent standard (real-valued, unless otherwise indicated) normal random variables. The  Euclidean length of the standard Gaussian random vector $G^{(d)}= (g_1, \ldots, g_d)$ in $\R^d$ is usually called the Chi distribution (with $d$ degrees of freedom) and often denoted by $\chi_d$.  It is well-known that its mean is  
\begin{equation} \label{mud}
  \mu_d := \E \|G^{(d)}\|_2 = \E (|g_1|^2 + \ldots + |g_d|^2)^{1/2} = \sqrt{2} \;\frac{\Gamma((d+1)/2)}{\Gamma(d/2)}    
\end{equation}
and that $\frac{\mu_d }{\sqrt{d}}$ increases to $1$ as $d \to \infty$. (See, e.g., Appendix A.2 in \cite{AS17}, but note that the quantity from \eqref{mud} is denoted there by $\kappa_d$.)

\section{The main results} \label{section2}

The main result of this note is the following.

\begin{theorem}\label{thm:bound-on-radius}
Suppose that $E$ is a subspace of $\R^N$ of dimension $d$ ($1 \le d \le N$).
\begin{enumerate}
\item[(i)]  If  \  $\lambda \|x\|_2 \le \|x\|_1$ for all  $x \in E$, 
then
\begin{equation} \label{lambda_up}
 \lambda \le \sqrt{\frac{2}{\pi}}  \sqrt{N}\times \frac{\sqrt{d}} {\mu_d}.  
\end{equation}
\item[(ii)] If  \   $\|x\|_1 \le \lambda' \|x\|_2$  for all  $x \in E$, then 
\begin{equation} \label{lambda_dn}
  \lambda' \ge 
   \sqrt{d} .
\end{equation}
\end{enumerate}
\end{theorem}
\begin{remark} \label{rem1}{\rm (a) It is part (i) of the Theorem that is of primary interest. Of course, once an upper bound is determined, it is natural to ask about a lower bound, and we found it intriguing that such estimate (without additional multiplicative constants) was not readily available in the literature.\\
(b) The assertion (i) simplifies a little if we replace the underlying counting measure by the normalized one, or by any probability measure. We get then, by the same argument, 
\begin{equation} \label{upperbound2}
\lambda \leq \sqrt{\frac{2}{\pi}}\times \frac{\sqrt{d}}{\mu_d}.
\end{equation} 
In the present setting, the bound \eqref{upperbound2} is optimal. For example, for $d=2$, the subspace of $L_1([0,1])$ spanned by $\sin(2\pi t)$ and $\cos(2\pi t)$ saturates the bound, which is $\frac{\sqrt{8}}\pi$. (Saturation for $\ell_1^N$ --- in the limit as $N\to \infty$ --- follows by discretization.) 
For general $d$, the space of linear functions 
on $L_1(S^{d-1})$ (with normalized surface measure) does the trick. Of course, the bound in (ii) is saturated if $E$ is a coordinate subspace. \\
(c) While the assertion (ii) just reiterates the point that if $d=\Omega(N)$, then the scaling constant $\lambda$ must be $\Omega(\sqrt{N})$, the assertion (i) is more tantalizing. It says that  the ratio $\lambda/\sqrt{N}$ {\em can not } be close to $1$ for {\em any } $d>1$ and, for large $d$, must satisfy $\lambda/\sqrt{N} \leq \sqrt{\frac{2}{\pi}} +o(1)$. Recall now that the standard proofs of Dvoretzky's theorem (e.g., \cite{FLM77}, which essentially follows \cite{Milman}) lead to the scaling constant that is 
the average (or the median) of the underlying norm over the Eucldean sphere. 
As is well known (see, e.g., Appendix A.2 in \cite{AS17}), in the case of the $\ell_1^N$-norm that average is precisely $\frac{N \mu_1}{\mu_N}$, which --- as $N\to \infty$ --- is of order 
$\sqrt{\frac{2}{\pi}} \sqrt{N}$, precisely the same as the upper bound given by our Theorem \ref{thm:bound-on-radius}(i). In other words, our Theorem  shows that such behavior is not an artifact of the method, but a necessary feature.\\
(d) Most of the assertions of Theorem \ref{thm:bound-on-radius} carry over to the complex case, and generalize to statements about subspaces of $\ell_p^N$ for $p\in (1,2)$. However, some of the bounds have to be (necessarily) slightly modified, and we do not have a completely satisfactory version of (ii) for $p>1$. We elaborate on these issues in the next section. } \qed
\end{remark}
\begin{proof}[Proof of Theorem \ref{thm:bound-on-radius} (i)] 
 Let $\{x_1, \ldots, x_d\}$ be an orthonormal basis in $\ell_2^N \cap E \subseteq \R^N$. Consider the random vector $X$ in $\R^N$ given by
\[
  X = \sum_{j=1}^d x_j g_j
\]
which has a Gaussian distribution on $E$. We compute the expectation of the $\ell_1$-norm of $X$. Using the property of the normal distribution that $\sum_{j=1}^d a_j g_j$ has the same distribution as $\big( \sum_{j=1}^d |a_j|^2 \big)^{1/2} g_1$,
\begin{equation} \label{eq:first-moment}
  \E \|X\|_1 = \sum_{k=1}^N \E \bigg| \sum_{j=1}^d x_j(k) g_j \bigg|
= \E|g_1| \sum_{k=1}^N \bigg( \sum_{j=1}^d |x_j(k)|^2 \bigg)^{1/2}.
\end{equation} 
Recall that $\E|g_1| = \mu_1= \sqrt{\frac{2}{\pi}}$. 
Applying Cauchy-Schwarz inequality we obtain
\[ 
  \E \|X\|_1 \le \sqrt{\frac{2}{\pi}} \bigg( \sum_{k=1}^N \sum_{j=1}^d |x_j(k)|^2 \bigg)^{1/2} \sqrt{N}
  = \sqrt{\frac{2}{\pi}} \sqrt{d N}.
\]
On the other hand, by~the assumption of (i),
\[
  \E \|X\|_1 \ge \lambda \E \bigg\| \sum_{j=1}^d x_j g_j \bigg\|_2
= \lambda \E \big( \sum_{j=1}^d |g_j|^2 \big)^{1/2} = \lambda \mu_d,
\]
where the first equality follows from the fact, that $x_1, \ldots, x_d$ are orthonormal in $\ell_2^N$. Combining the two above estimates for $\E \|X\|_1$ we arrive with the desired upper bound for $\lambda$.
\end{proof}

\begin{proof}[Proof of Theorem \ref{thm:bound-on-radius} (ii)]  
 While it is hard to believe that this bound is new, we could not find a reference in the literature. 
The proof is based on the following lemma. 
\begin{lemma} \label{lemma:norm12} {\rm (\cite{Pie67})}  
Let $i_{12} : \ell_1 \to \ell_2$ be the formal identity map. Then $\pi_2(i_{12}) =1$. 
Equivalently, if $T: \ell_2 \to \ell_1$ is a bounded linear operator, then $\|i_{12} \circ T \|_{\rm HS} \leq \|T\|$. 
\end{lemma}
Above $\pi_2(\cdot)$ is the $2$-absolutely summing norm and 
$\|A\|_{\rm HS}:= \big(\tr (AA^*)\big)^{1/2}$ --- the Hilbert-Schmidt norm of a Hilbert space operator.  
Since we will only use the second statement of the Lemma,  we refer the reader to 
Chapter 2 of \cite{nicole-book} for the definition and properties of $p$-absolutely summing and related norms. 
What we need to know here is that, for a normed space $X$ and an operator $S: X \to \ell_2$, 
we have  
$$\pi_2(S)= \sup \{\|ST\|_{\rm HS} \; :  \;  \|T\|\leq 1,\, T : \ell_2 \to X \},$$
which readily implies the equivalence between the two statements of the Lemma.

To prove assertion (ii) of the Theorem we apply Lemma \ref{lemma:norm12} with $T = P_E$, 
the orthogonal projection onto a $d$-dimensional subspace $E \subset \R^N$,
considered as an operator from $\ell_2$ to $\ell_1$; 
formally $T= i_{12} \circ P_E$.  Then the norm $\|P_E : \ell_2 \to \ell_1\|$ equals $\max \{\|x\|_1 :  x\in E, \|x\|_2 =1\}$, 
which is $\leq \lambda'$ by hypothesis. 
On the other hand, clearly $\|P_E\|_{\rm HS} =\sqrt{d}$ and the conclusion follows.  
\end{proof}
Lemma \ref{lemma:norm12} seems to have been forgotten --- or overseen --- by subsequent authors,   
who on occasions cited slightly weaker bounds. For example, in \cite{nicole-book} it is noted that 
$\pi_2(i_{12}) \leq \sqrt{2}$ (as a consequence of Khinchine's inequality, Proposition 10.6) 
and $\pi_2(i_{12}) \leq \sqrt{\frac{\pi}{2}}$
(as a consequence of the little Grothendieck inequality, Proposition 10.8). It is conceivable that the reason for this oversight was that most authors were focused on the much more interesting boundedness of the $1$-summing norm $\pi_1(i_{12})$, which is an instance of the main Grothendieck inequality (\cite{Gro53}; see, e.g., Proposition 10.6 and Theorem 10.11 in 
\cite{nicole-book}), and in comparison to which Lemma~\ref{lemma:norm12} might have seemed trivial. Another factor might have been the fact that the language of the original 1967 reference was German. In fact, we were made aware of \cite{Pie67} (thanks, Guillaume!) only after a preliminary version of the present paper was circulated. 

We include below a proof (in fact, three proofs) of Lemma \ref{lemma:norm12}. First, we do it for completeness, given the limited accessibility of the original paper \cite{Pie67}. More importantly, we were not able to obtain a satisfactory generalization of the Lemma (and, as a consequence, of Theorem \ref{thm:bound-on-radius}(ii)) to subspaces of $\ell_p$, $p\in (1,2)$;  supplying the proofs will allow us to state the missing ingredients as conjectures, see Section \ref{section3} for details. 
\begin{proof}[Proof of Lemma \ref{lemma:norm12}]  It is enough to address the finite-dimensional version, 
i.e., with $i_{12} : \ell_1^N \to \ell_2^N$ and $T: \ell_2^m \to \ell_1^N$ 
(in any case, this is the setting that is relevant here).
We need to prove that $\tr A \leq \|T\|^2$, where 
$A=TT^*$ is considered as an operator on $\ell_2$ (formally $A=i_{12}TT^*i_{12}^*=i_{12}TT^*i_{2\infty}$). 
Specifically, we will show that
\begin{equation} \label{trace}
   \tr A \leq \|A\|,
\end{equation}
where the norm on the right is $\|A : \ell_\infty^N \to \ell_1^N\|$.
Note that this is equivalent to the previous statement since 
$\|A\| = \|TT^\ast\| = \|T\|^2$.  (This is a standard identity for Hilbert space operators, but is also true if the domain of $T$ is a Hilbert space and the target space is arbitrary.)  
Denoting $A=\big(a_{ij}\big)$, 
we have 
\[
\|A : \ell_\infty^N \to \ell_1^N\| =\max \left\{  \sum_{i,j} a_{ij} \epsilon_i \epsilon_j' \, : \,   |\epsilon_i| = |\epsilon_j'| = 1 \textup{ for $i,j=1,\ldots,N$} \right\}. 
\]
On the other hand, 
\begin{eqnarray*}
\tr A = \sum_{i} a_{ii} = \E_{\bf \epsilon}  \sum_{i,j} a_{ij} \epsilon_i \epsilon_j\leq \max_{\bf \epsilon}  \sum_{i,j} a_{ij} \epsilon_i \epsilon_j , 
\end{eqnarray*}
with both the expectation and the maximum taken over all choices of signs ${\bf \epsilon}= (\epsilon_i)$, $|\epsilon_i| = 1$, 
and \eqref{trace} readily follows.  

Another presentation of the argument is as follows: 
\begin{eqnarray*}
\|T \colon \ell_2 \to \ell_1 \|^2
&=& \sup_{\|x\|_2 = 1} \sup_{\bf \epsilon} \big| \langle T x, {\bf \epsilon} \rangle \big|^2
= \sup_{\bf \epsilon} \sup_{\|x\|_2 = 1} \big| \langle x, T^\ast {\bf \epsilon} \rangle\big|^2 \\
&=& \sup_{\bf \epsilon} \|T^\ast {\bf \epsilon}\|_2^2 \ 
\ge \ \E_{\bf \epsilon} \langle T T^\ast {\bf \epsilon}, {\bf \epsilon} \rangle
= \tr T T^\ast .
\end{eqnarray*}

Finally, here is a direct ``high level'' proof that the $2$-integral norm of $i_{12}$ (and hence $\pi_2(i_{12})$) 
does not exceed $1$. We have 
\[
i_{12} = R \circ j_{\infty 2} \circ R^*
\]
where $j_{\infty 2} : L_\infty \to L_2$ is the formal identity, $R :  L_2 \to \ell_2$ is the (norm $1$) 
Rademacher projection (see \S10 of \cite{nicole-book}), and 
$R^* : \ell_1 \to L_\infty$ the isometric embedding (using the Rademacher functions); 
this yields the needed Pietsch-Grothendieck factorization of $i_{12}$ (cf. \S9  of \cite{nicole-book}). 
\end{proof} 

\begin{remark} \label{rem2} 
{\rm We conclude this section with a few   observations concerning  Theorem \ref{thm:bound-on-radius} and its proof. \\
(a) There are several alternative approaches to establishing (weaker) versions of bounds (i) and (ii) from Theorem \ref{thm:bound-on-radius}. For example, it is proved in \cite{meyer-pajor-sections} that, among all sections of $B_1^N$ (the unit ball in $\ell_1^N$) of given dimension, the coordinate sections have the maximal volume. Therefore, a direct calculation shows that --- under the hypothesis of (ii) --- we must have (cf. \cite{SzarekOnKashin, STJ80})
\[
  \lambda' \ge 
  \frac{\Vol_d(B_2^d)^{1/d}}{\Vol_d(B_1^d)^{1/d}}=
  \left(\frac{\pi ^{d/2} \Gamma (d+1)}{2^{d}  \Gamma \left(\frac{d}{2}+1\right)}\right)^{1/d} 
   \ge 
 \sqrt{\frac{\pi}{2e}} \ \sqrt{d},
\]
where $\Vol_d$ stands for the $d$-dimensional volume. This is not as precise as what our Theorem gives, but the argument is quite robust: the results of \cite{meyer-pajor-sections} are valid for other $p\in (1,2)$   and in the complex case.  \\
(b) The upper bound for $\lambda$ cannot be essentially improved even if we assume that~\eqref{eq:almost-euclidean-1} and~\eqref{eq:almost-euclidean-2} hold simultaneously. Indeed, as pointed out in Remark \ref{rem1}(c),  
the bound given by \eqref{lambda_up} is approximately saturated by the original argument from \cite{FLM77}, which yields both~\eqref{eq:almost-euclidean-1} and~\eqref{eq:almost-euclidean-2}.  
}
\end{remark}

\section{The complex case and the range $1 < p < 2$} \label{section3}

The assertion (i) of Theorem~\ref{thm:bound-on-radius} can be easily extended to subspaces of $\ell_p^N$ over real or complex scalar field, for any $p \in [1,2]$. In what follows, the $L_p$-norm of the Euclidean length of the standard Gaussian random vector $G^{(d)}$ in $\R^d$ will be denoted by $\mu_{d,p}$. A standard calculation shows that
\[
  \mu_{d,p} := \left( \E \|G^{(d)}\|_2^p \right)^{\frac1p} = 
  \sqrt{2} \left( \frac{\Gamma\big((d+p)/2\big)}{\Gamma(d/2)} \right)^{\frac1p}.
\]
Similarly to the case $p=1$, $\frac{\mu_{d,p}}{\sqrt{d}}$ increases to 1 as $d \to \infty$ for any $1 \le p < 2$.
\begin{theorem}\label{thm:general-p-complex-case}
Let $p \in [1, 2]$ and suppose that $E$ is a subspace of $\R^N$ or $\C^N$ of dimension $d$, where $1 \le d \le N$. If $\lambda \|x\|_2 \le \|x\|_p$ for all $x \in E$, then
\begin{equation}\label{1<p<2}
  \lambda \le N^{\frac1p - \frac12} \times \sqrt{d} \times \begin{cases}
    \frac{\mu_{1,p}}{\mu_{d,p}} & \text{in the real case,} \\[1ex]
    \frac{\mu_{2,p}}{\mu_{2d,p}} & \text{in the complex case.}
  \end{cases}
\end{equation} 
\end{theorem}
\begin{remark} \label{rem3}
{\rm
(a) In the complex case and for $p = 1$,
\[
  \lambda \le \frac{\sqrt{\pi}}{2} \sqrt{N} \times \frac{\sqrt{2d}}{\mu_{2d}}
  = \left(\frac{\sqrt{\pi}}{2}+o(1)\right) \sqrt{N}.
\]
(b) The bound \eqref{1<p<2} is nontrivial (i.e., the factor standing by $N^{\frac1p - \frac12}$ is strictly smaller than $1$) for all $d\geq 2$ and all $p\in [1,2)$.
}
\end{remark}
\begin{proof}[Proof of Theorem~\ref{thm:general-p-complex-case}] The proof follows the same lines as that of Theorem~\ref{thm:bound-on-radius}(i), only a few minor changes are needed. First, instead of bounding $\E \|X\|_1$, one should consider $\E \|X\|_p^p$. Next, the Cauchy-Schwarz inequality is to be replaced by H{\"o}lder's inequality. Finally, in the complex case, $g_1, g_2, \ldots$ are independent standard complex Gaussian random variables. 
\end{proof}
Let us now discuss possible extensions of the assertion (ii) of Theorem~\ref{thm:bound-on-radius}. For $p = 1$ and the complex case, the result remains valid as does Lemma~\ref{lemma:norm12}. The proof of Lemma~\ref{lemma:norm12} in the complex case is essentially the same as in the real case. The argument based on bounding from below the $\ell_\infty^N \to \ell_1^N$ norm of $A = T T^\ast$ requires only some minor formal changes: one can take the real value of the bi-linear form $\sum_{i,j} a_{ij} \epsilon_i \epsilon'_j$; however, $\epsilon_i$, $\epsilon'_j$ can remain $\pm 1$, since we do not actually need to saturate the norm of $A$ with $\epsilon$'s ranging over the whole unit circle. The alternative argument for the real case of Lemma~\ref{lemma:norm12} carries over to the complex case without any changes.

 For $1 < p < 2$ we do not have a fully satisfactory counterpart to the assertion (ii).\footnote{Most of the results of this paper were obtained in 2013 and shortly thereafter, but we procrastinated with publication, mostly because we hoped to fully resolve the $1<p<2$ case.} As already mentioned in Remark \ref{rem2}(a), using the result of Meyer and Pajor~\cite{meyer-pajor-sections}, one deduces that 
 \begin{equation} \label{pii}
\lambda' \ge c\hskip0.5mm d^{1/p - 1/2},
 \end{equation} 
 and there are several other lines of argument that yield a similar conclusion. We conjecture that inequality \eqref{pii} holds with $c=1$.  However, a proof appealing to \cite{meyer-pajor-sections} (i.e. via comparison of volumes of $B_p^d$- and $B_2^d$-balls) is clearly  suboptimal;  this is true already in the case $p=1$, see Remark \ref{rem2}(a).
 A natural way to generalize Theorem~\ref{thm:bound-on-radius}(ii) to $1 < p \le 2$ with a sharp estimate on $\lambda'$ would be to find a counterpart for Lemma~\ref{lemma:norm12}. Below we propose a few statements as problems/conjectures; they correspond to the different proofs of Lemma~\ref{lemma:norm12} we presented in Section \ref{section2}:

\begin{conjecture} \label{conj}
\smallskip 
{\rm (a)}  $\pi_{r,2} (i_{p2}) \leq 1$ for $p \in (1,2)$, where $i_{p2} : \ell_p \to \ell_2$ is the formal identity map, 
$\pi_{r,2}(\cdot)$ is the $(r,2)$-absolutely summing norm (see Chapter 2 of \cite{nicole-book}), and $r = 2p/(2-p)$. 

\smallskip \noindent
{\rm (b)} $\|B \colon \ell_{2} \to \ell_p\| \ge \|B\|_{S_{r}}$ with $r=2p/(2-p)$, where $\|A\|_{S_{r}} = (\tr (A^*A)^{r/2})^{1/r}$ is the $r$-Schatten norm. 

\smallskip \noindent
{\rm (c)} If $A$ is positive semi-definite, then $\|A \colon \ell_{p'} \to \ell_p\| \ge \|A\|_{S_{r}}$ with $r=p/(2-p)$ (and, again, $p' = p/(p-1)$). 
\end{conjecture}
 
 \smallskip \noindent Confirming any of these conjectures would imply a generalization of Theorem~\ref{thm:bound-on-radius}(ii) to $p\in (1,2)$ with the (optimal) assertion $\lambda' \ge d^{1/p - 1/2}$. 
 (Actually, like in the case of the different reformulations of Lemma~\ref{lemma:norm12}, the statements (a)--(c) are formally equivalent.)
 Note that the dimension doesn't  explicitly enter any of the statements (a)--(c) above, they are just assertions about various operator ideal norms and/or (real or complex) matrices of arbitrary size. It is also conceivable that there is a simple direct proof of the inequality \eqref{pii} with $c=1$, either via change of density (see Section 1.2 in \cite{JS01}) or by some fancy interpolation argument...

\section{Survey of similar questions,  open problems} \label{section4}

Similar questions/results for other normed spaces, most often $\ell_p$-spaces or their non-commutative analogues, Schatten spaces, were found relevant to problems in various areas of mathematics and its applications. Here is a sampler of the connections and possible questions. 

Let us start with the spaces $\ell_p^N$ for $p\in (2, \infty)$ (and, say, $p\leq \log N$). In that setting, it is more convenient to use the normalized counting measure already mentioned in Remark \ref{rem1}(b).  
The proof from \cite{FLM77} supplies then almost-Euclidean subspaces that (ignoring a multiplicative $c(\e)$ constant) are of dimension $d=\Theta(p N^{\frac 2p})$ with the scaling constant $\lambda = \Theta(\sqrt{p})$ (see, e.g., Section 7.2.4.1 in \cite{AS17} for  reasonably detailed yet accessible calculations). It was further shown in \cite{FLM77} that the obtained dimension $d$ is asymptotically optimal, but the argument doesn't yield any constraints on the scaling constant. It is likely that some information in this regard can be obtained via the approach applied in the preceding section for $p\in [1,2)$. Likewise, 
the range $p\in (0,1)$ can be similarly investigated. 
(This range falls outside of the familiar context of convex geometry, but is of significant interest in applications, for example in compressed sensing; see, e.g., \cite{Don06}.) 
We leave these questions as exercises for an interested reader, but will instead mention two circles of ideas that touch on algorithmic aspects of the problem. 

First, if (and only if) $p\in 2 \N$, there {\em do exist} sections (of the unit ball $B_p^N$) that are {\em precisely} spherical (see \cite{Mil88, Kon95, KK01}). The constructions have connections to number theory, numerical analysis, coding theory, spherical designs, and combinatorics, and are completely explicit. If $p=4$, it is possible to achieve $d$ and $\lambda$ matching --- up to universal multiplicative constants --- the values obtained by the probabilistic proof. If $p=6,8,\ldots$, the existing explicit constructions are suboptimal (according to \cite{KK01}). In any case, it would be interesting to determine the precise asymptotics (for fixed $p>2$, and as $p\to \infty$) of the optimal scaling constant $\lambda$. 

Other examples of nearly-Euclidean subspaces of $L_p$ come from harmonic analysis. Let $L_p=L_p([0,2\pi])$ (with normalized Lebesgue measure) and, for $k \in \Z$, consider the functions $\chi_k(t) = e^{ikt}$, which can be thought of as characters on the circle group $\T$. For $p>2$, a set $S \subset \Z$ is called a $\Lambda_p$-set if, for some $\lambda \geq 1$, the inequality 
$\|\sum_{k\in S} t_k \chi_k\|_p \leq \lambda (\sum_{k\in S} |t_k|^2)^{1/2}$ 
holds for any sequence of scalars $(t_k)$. (For $p<2$, the inequality has to be reversed.) Since the converse inequality holds always, without any multiplicative constant, this means that the subspace of $L_p$ spanned by $(\chi_k)_{k\in S}$ is nearly-Euclidean. Again, if $p\in 2 \N$, there are explicit (based on combinatorics and finite geometries) examples of sets $S \subset \{1,2,\ldots,N\}$ of cardinality $\Theta(N^{\frac 2p})$ such that $S$ is a $\Lambda_p$-set with $\lambda=\Theta(\sqrt{p})$ (\cite{Rudin}). For other $p\in (2,\infty)$, subsets $S$ of cardinality $\Theta(N^{\frac 2p})$ and with $\lambda\leq C(p)$ can be found using rather sophisticated probabilistic methods (see \cite{Bou89,Talagrand,Bou01}). It would be very interesting to produce such $\Lambda_p$-sets that are either explicit, or generated by a feasible algorithms, particularly for $p\in(2,4)$. 

It may appear to the reader that the last example is ``out of character'' in that it deals with continuous setting, functions on an interval rather than $\ell_p^N$. However, this discordance is easily remedied, either by discretization, or by working with characters on the group $\Z_N$ (or, say, $\Z_2^m$, $m=\log_2 N$) to begin with. 

Finally, it should be mentioned that 
$\Lambda_p$-sets of extremal size are not possible if $p<2$. Indeed, if $S \subset \{1,2,\ldots,N\}$ satisfies $\#S = \Omega(N)$, then --- by Szemer\'edi's theorem \cite{Sze75} --- $S$ contains, for large $N$, long arithmetic progressions, which precludes it from being a $\Lambda_p$-set for $p\neq 2$. (An analogous argument is even simpler for groups such as $\Z_2^m$.) On the other hand, for $p\in (0,2)$, $\Lambda_p$-sets with $\#S = \Omega(N^\alpha)$, $\alpha<1$, exist as a consequence of the corresponding results for $p>2$ (and the H\"older inequality), but note that the latter were random constructions. 

Pecularities of scaling constants for almost- and nearly-Euclidean subspaces of $\ell_\infty^N$ were exploited in \cite{KS03} to solve (in the negative) the then almost $60$-years-old Knaster problem : {\em Given continuous function $f : S^{n-1} \to \R$ and a configuration $\{x_k\}$ of $n$ points on $S^{n-1}$, does there exist an isometry $V\in O(n)$ such that $f(Vx_k)$ is constant?} In a nutshell, the idea behind \cite{KS03} was that, for an almost-Euclidean subspaces of $\ell_\infty^N$, the scaling constant should satisfy an inequality of the kind $\sqrt{ d/N} \leq \lambda \leq C(\e) + C\sqrt{ d/N}$. This was hinted, but not precisely stated and only partially proved in \cite{KS03} (see Lemmas 3 and 4 in that paper; note that for the upper bound on $\lambda$ to hold, it may be necessary to discard ``superfluous'' coordinates). It would be interesting to complete this analysis. 

It is intriguing that, as pointed in \cite{Mil88}, a positive answer to the Knaster problem would yield a proof of Dvoretzky's theorem with optimal dependence on the parameter $\e$. While the solution in \cite{KS03} was negative, it doesn't preclude the possibility of a positive answer if the size of the configuration is noticeably smaller than $n$, say $\Omega(n^{1/2})$, which would be still sufficient for the application to Dvoretzky's theorem.


Let us now turn our attention to the non-commutative $L_p$'s, the Schatten spaces $S_p=S_p^N$. It will be more convenient to work with the version of Schatten norms induced by the normalized trace. 
Further, to lighten the notation, for the remainder of this paper we will denote the $p$-Schatten norm by $\|\cdot\|_p$. That is, for an $N\times N$ matrix $A$ we will write $\|A\|_p = \left(\frac 1N \tr (A^*A)^{p/2}\right)^{1/p}$. (The limit $\|\cdot\|_\infty$ norm is the usual operator norm.) 

If $p=\infty$, we have a minor miracle. Namely, if we restrict our attention to the real vector space of Hermitian matrices, and $N=2^m$, then $S_\infty^N$ contains a subspace of $E$ of dimension $d = \Theta(m)$, on which all $p$-norms coincide or, equivalently, every $A \in E$ is a multiple of an isometry. That is, the scaling constant $\lambda$ is precisely $1$.  (Strangely, this phenomenon can not be reproduced in the complex setting; see, e.g., Section 11.1 in \cite{AS17} for more details about both settings.) This ``miracle'' allowed Tsirelson (\cite{Tsi80}) to relate the Grothendieck inequality to quantum violations in Bell inequalities; see Chapter 11 in \cite{AS17} for more background about this circle of ideas and Problem 11.27 there for a more general question about the scaling constant in the present context. 

For $p\in [1,\infty)$, the method of \cite{FLM77} yields optimal (up to multiplicative constants depending at most on $p$ and $\e$) values of the dimension of almost- and nearly-Euclidean subspaces of $S_p^N$, and gives information about the resulting scaling constant. The estimates can be tightened by more careful calculations (see, e.g., Section 7.2.4.2 and Exercises 7.23\hskip0.75mm \&\hskip 0.5mm 7.25 in \cite{AS17}) and further studied using methods parallel to Sections \ref{section2} and  \ref{section3} of the present paper, but we will not engage in such analysis here. What we would like to emphasize is that these instances of Dvoretzky's theorem, and in particular the values of the scaling constants, turned out to be highly relevant to the solutions of important additivity problems for quantum channels \cite{holevo-icm, AHW} in quantum information theory. 
Specifically, the cases $p\in (2,\infty)$ were relevant to the study of R\'enyi entropy \cite{HW,ASW10}, while a more subtle analysis of (a version of) the case $p=4$ led to results about von Neumann entropy \cite{hastings, ASW11}. 
Unfortunately, given that these results were based on random versions of Dvoretzky's theorem, we do not actually have good explicit examples of channels which yield ``quantum advantage.'' (The best result in this direction seems to be \cite{GHP}, and it is clearly just a starting point.) 

 It would be extremely interesting to find explicit (or generated by a feasible algorithm) examples of nearly-Euclidean subspaces of the Schatten space $S_p=S_p^N$ (particularly for $p\in(2,4]$), whose dimensions and scaling constants reasonably  approach those obtained by the random method. 
Unfortunately, there are no obvious non-commutative analogues of spherical designs or $\Lambda_p$-sets that do the job, in spite of considerable research on related topics such as unitary $t$-designs, see, e.g., \cite{H+23}.

\bigskip\noindent{\em Acknowledgements } The authors thank the anonymous referee for detailed comments and Guillaume Aubrun for bringing reference \cite{Pie67} to their attention. The research of the first-named author has been partially supported by grants from the National Science Foundation (U.S.A).


\vskip1cm

\end{document}